\documentclass[a4paper,11pt]{article}
\usepackage[english]{babel}
\usepackage[latin1]{inputenc}
\usepackage{amssymb}
\usepackage{amscd}
\usepackage{mathtools}
\usepackage[paper=a4paper,left=3cm,right=3cm,top=3cm,bottom=2.3cm]{geometry}
\usepackage[thmmarks,amsmath,amsthm]{ntheorem}

\usepackage{hyperref}

\newtheorem{thm}{Theorem}[section]
\newtheorem{lem}[thm]{Lemma}
\newtheorem{prop}[thm]{Proposition}
\newtheorem{cor}[thm]{Corollary}
\newtheorem{rem}[thm]{Remark}
\newtheorem{exam}[thm]{Example}

\begin{document}

\title{Vector bundles and Arakelov geometry on the projective line over the integers}
\author{Fabian Reede\footnote{Georg-August-Universit\"at G\"ottingen, freede@uni-math.gwdg.de}}
\maketitle

\begin{abstract}
\noindent We study locally free sheaves of rank two on $\mathbb{P}^1_{\mathbb{Z}}$, especially indecomposable ones. Subsequently we apply various concepts of Arakelov geometry to these sheaves. We compute for example the arithmetic Chern classes and use the arithmetic Riemann-Roch theorem.
\end{abstract}
\section*{Introduction}
It is a well known theorem that every locally free sheaf $\mathcal{E}$ of rank $r$ on $\mathbb{P}^1_k$ is a direct sum of $r$ line bundles, here $k$ is an arbitrary field. That is, there are $r$ locally free sheaves $L_i$ of rank one and an isomorphism 
\begin{center}
$\mathcal{E}\cong\bigoplus\limits_{i=0}^r L_i$.
\end{center} 
But if one replaces the field $k$ with the ring $\mathbb{Z}$, then this theorem fails. In (\cite{roberts}) Roberts constructs indecomposable locally free sheaves on $\mathbb{P}^1_{\mathbb{Z}}$ of arbitrary rank $r$ using transition matrices with respect to an open cover of $\mathbb{P}^1_{\mathbb{Z}}$.\\
In the first chapter we study locally free sheaves of rank two on $\mathbb{P}^1_{\mathbb{Z}}$. We prove that every such sheaf can be written as an extension of a line bundle by a twisted sheaf of ideals. This sheaf of ideals (also called ideal sheaf) is defined by a local complete intersection of codimension two in $\mathbb{P}^1_{\mathbb{Z}}$.\\ 
We prove that a locally free sheaf $\mathcal{E}$ of rank two is indecomposable if the local complete intersection is not empty. This gives a method of constructing indecomposable locally free sheaves of rank two on $\mathbb{P}^1_{\mathbb{Z}}$ globally, without having to use an open cover.\\ 
We do all this by a careful analysis of the behaviour of the ideal sheaves on the fibers of the morphism $\mathbb{P}^1_{\mathbb{Z}}\rightarrow \mathbb{Z}$. We finish this part by studying a simple example.\\
In the last part of this chapter we study the cohomology groups of twisted ideal sheaves to get a better understanding of the cohomology groups of the locally free sheaf $\mathcal{E}$ via the long exact sequence of cohomology. As an example we describe locally free sheaves of rank two without cohomology.
\medskip

\noindent Since $\mathbb{P}^1_{\mathbb{Z}}\rightarrow \mathbb{Z}$ is an arithmetic surface, we can use the theory of Arakelov geometry to study further properties and invariants of these locally free sheaves which are not available since we replaced the field $k$ by $\mathbb{Z}$.\\ 
We start by computing the arithmetic Chern classes $\widehat{c}_i(\overline{\mathcal{E}})$ of a locally free sheaf $\mathcal{E}$ of rank two. In order to do this we need to define the Hermitian vector bundle $\overline{\mathcal{E}}$. This is the locally free sheaf $\mathcal{E}$ with a Hermitian metric $h$ on the associated vector bundle $E$ on the manifold $\mathbb{P}^1_{\mathbb{Z}}(\mathbb{C})=\mathbb{C}\mathbb{P}^1$. By using the fact that $\mathcal{E}$ can be written as an extension we can explicitly describe the metric we want to use for our purposes.\\ 
We go on and recall the arithmetic Riemann-Roch theorem due to Gillet and Soul\'{e}. As an application we derive the arithmetic Hirzebruch-Riemann-Roch theorem for a locally free sheaf of rank two on $\mathbb{P}^1_{\mathbb{Z}}$. Other applications include the computation of the Ray-Singer analytic torsion of the line bundle $\mathcal{O}_{\mathbb{P}^1_{\mathbb{Z}}}(a)$ for $a\in \mathbb{Z}$ and the computation of an arithmetic (Arakelov) analogue of the usual Euler characteristic.
\medskip

\noindent This article grew out of the author's wish to understand the theory of vector bundles in Arakelov geometry. Since much of the existing literature on this issue is written in broad generality, the author tried to understand this theory by looking at the simplest nontrivial example and computing everything explicitly. This article should give a taste of the techniques and computations used in Arakelov geometry. Another aspect was to get a rough overview of the existing literature.
\subsection*{Notation}
In the following $X$ denotes the scheme $\mathbb{P}^1_{\mathbb{Z}}$ and $Y$ the scheme $Spec(\mathbb{Z})$. Furthermore $f: X\rightarrow Y$ denotes the structural morphism.\\ 
For $p\in Y$ we denote the fiber $f^{-1}(p)\subset X$ by $X_p$. We have $X_p=\mathbb{P}^1_{\mathbb{F}_p}$. The fiber $X_0$ over the generic point $0 \in Y$ is given by $\mathbb{P}^1_{\mathbb{Q}}$.\\ 
Similarly if $\mathcal{F}$ is a coherent sheaf on $X$ we denote the induced sheaf on $X_p$ by $\mathcal{F}_p:=\iota_p^{*}\mathcal{F}$ where $\iota_p: X_p\rightarrow X$ is the closed immersion of the special fiber $X_p$ for $p\neq 0$. All sheaves are supposed to be coherent if not otherwise stated.\\ 
If $Z$ is a local complete intersection of codimension two in $X$ we denote by $\mathcal{I}_Z$ the associated sheaf of ideals in $\mathcal{O}_X$.\\ 
All Hermitian metrics are supposed to be invariant with respect to the complex conjugation on $\mathbb{C}$, see (\cite[Definition IV.4.1.4.]{soul3}).
\tableofcontents

\section{Locally free sheaves on the projective line over the integers}
\subsection{Locally free sheaves of rank two}
Let $\mathcal{E}$ be a locally free sheaf of rank two on $X$. We want to understand $\mathcal{E}$ as an extension of sheaves. To do this we need the following proposition:
\begin{prop}[\normalfont{\cite[Proposition 1.3.1.]{mori2}}]\label{subsh}
Let $W$ be an irreducible noetherian integral scheme with generic point $\eta$ and let $K=\mathcal{O}_{W,\eta}$ be the function field of $W$. If $\mathcal{E}$ is a torsion free coherent sheaf on $W$ we denote by $\Sigma(W,\mathcal{E})$ the set of all saturated $\mathcal{O}_W$-subsheaves of $\mathcal{E}$. Furthermore denote by $\Sigma(K,\mathcal{E}_\eta)$ the set of all vector subspaces of $\mathcal{E}_\eta$. Then the map $\gamma_{\mathcal{E}}: \Sigma(W,\mathcal{E}) \rightarrow \Sigma(K,\mathcal{E}_\eta)$ given by $\gamma_{\mathcal{E}}(\mathcal{F})=\mathcal{F}_\eta$ is bijective.
\end{prop}
\begin{rem}
\normalfont A subsheaf $\mathcal{F}$ of $\mathcal{E}$ is called saturated if the quotient $\mathcal{E}/\mathcal{F}$ is torsion free.
\end{rem}
\begin{rem}
\normalfont The inverse map can also be described easily: for a vector subspace $W$ in $\mathcal{E}_{\eta}$ we define the saturated subsheaf $\mathcal{F}$ associated to W by $\mathcal{F}(U):=\mathcal{E}(U)\cap W$ for every Zariski open subset $U$ in $X$.
\end{rem}
\begin{lem}
Assume $\mathcal{E}$ is a locally free sheaf of rank two on X, then there is an exact sequence
\begin{equation}\begin{CD}\label{exseq1}
0 @>>> \mathcal{O}_X(a) @>>> \mathcal{E} @>>> \mathcal{I}_Z(b) @>>> 0.
\end{CD}\end{equation}
Here $Z$ is a local complete intersection of codimension two in $X$ and $a,b\in \mathbb{Z}$ with $a\geq b$.
\end{lem}
\begin{proof}
Using Grothendieck's theorem for locally free sheaves on the generic fiber $\mathbb{P}^1_{\mathbb{Q}}$,  we can write $\mathcal{E}_0=\mathcal{O}_{X_0}(a)\oplus\mathcal{O}_{X_0}(b)$ for two integers $a,b$ such that $a\geq b$, see for example (\cite[Theorem 4.1]{haze}).\\ 
We find that $\mathcal{O}_{X_0}(a)$ is a saturated subsheaf of rank one in $\mathcal{E}_0$. Thus we get a one-dimensional vector subspace $U:=\gamma_{\mathcal{E}_0}(\mathcal{O}_{X_0}(a))$ in $V:=\gamma_{\mathcal{E}_0}(\mathcal{E}_{0})$ using the previous proposition (\ref{subsh}). But via $i_0: X_0\rightarrow X$ the generic point $\eta \in X_0$ is also the generic point of $X$ and we have $K=\mathcal{O}_{X,\eta}=\mathcal{O}_{X_0,\eta}$ and $\gamma_{\mathcal{E}}(\mathcal{E})=V$.\\ 
Now we define $\mathcal{L}:=\gamma_{\mathcal{E}}^{-1}(U)$, then $\mathcal{L}$ is a saturated subsheaf of rank one in $\mathcal{E}$. As $\mathcal{E}$ is reflexive we see that $\mathcal{L}$ is also reflexive, see \cite[Lemma 1.1.16.]{oko}, and hence $\mathcal{L}$ is a line bundle as $X$ is a regular scheme of dimension two. Since $\mathcal{L}$ is saturated $\mathcal{Q}:=\mathcal{E}/\mathcal{L}$ is torsion free. This gives the exact sequence 
\begin{equation*}\begin{CD}
0 @>>> \mathcal{L} @>>> \mathcal{E} @>>> \mathcal{Q} @>>> 0.
\end{CD}\end{equation*}
Restricting to the generic fiber we see that we must have $\mathcal{L}_0\cong \mathcal{O}_{X_0}(a)$. But in this case the restriction map defines an isomorphism $\mathbb{Z}\cong Pic(X)\stackrel{\sim}{\longrightarrow} Pic(X_0)\cong \mathbb{Z}$ since all special fibers $X_p$ are principal. We can conclude $\mathcal{L}\cong\mathcal{O}_X(a)$.  Since $\mathcal{Q}$ is torsion free, there is a local complete intersection $Z$ of codimension two in $X$ and some $n\in \mathbb{Z}$ such that $\mathcal{Q}\cong\mathcal{I}_Z(n)$. Looking again at the generic fiber, $\mathcal{E}_0/\mathcal{O}_{X_0}(a)\cong \mathcal{O}_{X_0}(b)$, we see that we must have $n=b$.
\end{proof}
This shows that every locally free sheaf $\mathcal{E}$ of rank two on $X$ defines a triple $(a,b,Z)$ with $a,b\in \mathbb{Z}$ such that $a\geq b$ and $Z$ a local complete intersection of codimension two in $X$.\\ 
Now we are interested in the inverse problem: given such a triple $(a,b,Z)$ is there a locally free sheaf $\mathcal{E}$ of rank two sitting in an exact sequence ($\ref{exseq1}$)?\\
These sequences are classified by $Ext^1_{\mathcal{O}_X}(\mathcal{I}_Z(b),\mathcal{O}_X(a))$. Using the local-to-global spectral sequence we get a sequence
\begin{equation*}\begin{CD}
0 @>>> H^1(X,\mathcal{H}om_{\mathcal{O}_X}(\mathcal{I}_Z(b),\mathcal{O}_X(a))) @>>> Ext^1_{\mathcal{O}_X}(\mathcal{I}_Z(b),\mathcal{O}_X(a))\\ @>>> H^0(X,\mathcal{E}xt^1_{\mathcal{O}_X}(\mathcal{I}_Z(b),\mathcal{O}_X(a))) @>>> H^2(X,\mathcal{H}om_{\mathcal{O}_X}(\mathcal{I}_Z(b),\mathcal{O}_X(a))).
\end{CD}\end{equation*}
\begin{lem}\label{compl}
Let $Z$ be a local complete intersection of codimension two in $X$, then we have $\mathcal{E}xt^i_{\mathcal{O}_X}(\mathcal{O}_Z,\mathcal{O}_X)=0$ for $i=0,1$ and $\mathcal{E}xt^2_{\mathcal{O}_X}(\mathcal{O}_Z,\mathcal{O}_X)\cong\mathcal{O}_Z$. 
\end{lem}
\begin{proof}
Since $Z$ is a local complete intersection, $\mathcal{I}_Z$ is locally generated by a regular sequence. For $i=0,1$ this lemma is a computation using the Koszul complex of the regular sequence, see for example \cite[V.3,p.690]{grif}. For $i=2$ we have the local fundamental isomorphism, see for example (\cite[Theorem I.5.1.1.]{oko}):
\begin{center} 
$\mathcal{E}xt_{\mathcal{O}_X}^2(\mathcal{O}_Z,\mathcal{O}_X)\cong \mathcal{H}om_{\mathcal{O}_Z}(det(\mathcal{I}_Z/\mathcal{I}_Z^2),\mathcal{O}_Z)$. 
\end{center}
We can choose an isomorphism $det(\mathcal{I}_Z/\mathcal{I}_Z^2)\cong\mathcal{O}_Z$, since $Z$ is a zero-dimensional scheme. This shows that we have $\mathcal{E}xt^2_{\mathcal{O}_X}(\mathcal{O}_Z,\mathcal{O}_X)\cong\mathcal{O}_Z$.
\end{proof}
\begin{cor}
We have $H^i(X,\mathcal{H}om_{\mathcal{O}_X}(\mathcal{I}_Z(b),\mathcal{O}_X(a))=0$ for $i=1,2$.
\end{cor}
\begin{proof}
Applying $\mathcal{H}om_{\mathcal{O}_X}(-,\mathcal{O}_X)$ to the exact sequence
\begin{equation*}\begin{CD}
0 @>>> \mathcal{I}_Z @>>> \mathcal{O}_X @>>> \mathcal{O}_{Z} @>>> 0
\end{CD}\end{equation*}
and using the previous lemma shows that we have an isomorphism 
\begin{center}
$\mathcal{H}om_{\mathcal{O}_X}(\mathcal{I}_Z,\mathcal{O}_X)\cong \mathcal{H}om_{\mathcal{O}_X}(\mathcal{O}_X,\mathcal{O}_X)=\mathcal{O}_X$. 
\end{center}
This implies $\mathcal{H}om_{\mathcal{O}_X}(\mathcal{I}_Z(b),\mathcal{O}_X(a))\cong\mathcal{O}_X(a-b)$. As $a-b\geq 0$ we get the desired result for $i=1$. For $i=2$ we use that $H^2(X_p,\mathcal{O}_{X_p}(a-b))=0$ for any $p\in Y$ as $dim(X_p)=1$. Using Grauert's base change theorem, see (\cite[Corollary III.12.9]{hart}), yields the desired result.
\end{proof}
These facts lead us to the following:
\begin{lem}[\normalfont{\cite[Corollary 2.10.]{fried}}]
Given a triple $(a,b,Z)$, then there exists a locally free sheaf $\mathcal{E}$ of rank two on $X$ sitting in an exact sequence
\begin{equation*}\begin{CD}
0 @>>> \mathcal{O}_X(a) @>>> \mathcal{E} @>>> \mathcal{I}_Z(b) @>>> 0.
\end{CD}\end{equation*} 
\end{lem}
\begin{rem}
\normalfont Since $\mathcal{O}_X$ is locally free, the proof also shows that there is an isomorphism $\mathcal{E}xt^1_{\mathcal{O}_X}(\mathcal{I}_Z,\mathcal{O}_X)\cong\mathcal{O}_Z$. Indeed, as $\mathcal{O}_X$ is locally free, the long exact sequence shows that we have an isomorphism $\mathcal{E}xt^1_{\mathcal{O}_X}(\mathcal{I}_Z,\mathcal{O}_X)\cong\mathcal{E}xt^2_{\mathcal{O}_X}(\mathcal{O}_Z,\mathcal{O}_X)$ and the latter is isomorphic to $\mathcal{O}_Z$ by (\ref{compl}). As $Z$ has codimension two in $X$ we conclude that there is an isomorphism $H^0(X,\mathcal{E}xt^1_{\mathcal{O}_X}(\mathcal{I}_Z(b),\mathcal{O}_X(a)))\cong H^0(Z,\mathcal{O}_Z)$.
\end{rem}
The question remains if we can say which extensions give rise to a locally free sheaf.\\ 
Using the local-to-global spectral sequence we get an isomorphism 
\begin{center}
$Ext^1_{\mathcal{O}_X}(\mathcal{I}_Z(b),\mathcal{O}_X(a))\cong H^0(Z,\mathcal{O}_Z)$.
\end{center}
Then by (\cite[Theorem 2.8.]{fried}) an extension $\xi$ whose image under this ismorphism is a unit in $\mathcal{O}_Z$, that is a unit in $\mathcal{O}_{Z,z}$ for every $z\in Z$, will give rise to a locally free sheaf $\mathcal{E}$.
\subsection{Ideal sheaves on the fibers}
Given a local complete intersection $Z$ of codimension two in $X$, then there is the associated ideal sheaf $\mathcal{I}_Z$ in $\mathcal{O}_X$. We want to understand how $\mathcal{I}_Z$ behaves on the fibers of $f$, that is: what is $\iota_p^{*}\mathcal{I}_Z$? To do this we note that $f(Z)\subset Y$ consits of finitely many primes, that is $f(Z)=\left\{p_1,\ldots,p_r\right\}$ and we consider the cases $p\in f(Z)$ and $p\notin f(Z)$ separately.
\begin{lem}\label{idstr}
Let $p\in Y$ be such that $p\notin f(Z)$, then we have $\iota_p^{*}\mathcal{I}_Z= \mathcal{O}_{X_p}$.
\end{lem}
\begin{proof}
Using the exact sequence
\begin{equation*}\begin{CD}
0 @>>> \mathcal{I}_Z @>>> \mathcal{O}_X @>>> \mathcal{O}_{Z} @>>> 0
\end{CD}\end{equation*}
and the definition of $\iota_p^{*}$ we get the exact sequence
\begin{equation}\begin{CD}\label{pullb}
0 @>>> \mathcal{T}or_1^{\iota_p^{-1}\mathcal{O}_X}(\iota_p^{-1}\mathcal{O}_Z,\mathcal{O}_{X_p}) @>>> \iota_p^{*}\mathcal{I}_Z @>>> \mathcal{O}_{X_p} @>>> \iota_p^{*}\mathcal{O}_Z @>>> 0.
\end{CD}\end{equation}
Now $p\notin f(Z)$ is equivalent to $Z\cap X_p=\emptyset$, so that for every $x\in X_p$ we have $\iota_p(x)\notin Z$. This implies $(\iota_p^{-1}\mathcal{O}_Z)_x=\mathcal{O}_{Z,\iota_p(x)}=0$ for all $x\in X_p$.\\
So the $\mathcal{O}_{X_p}$-modules $\iota_p^{*}\mathcal{O}_Z=\iota_p^{-1}\mathcal{O}_Z\otimes_{\iota_p^{-1}\mathcal{O}_X}\mathcal{O}_{X_p}$ and $\mathcal{T}or_1^{\iota_p^{-1}\mathcal{O}_X}(\iota_p^{-1}\mathcal{O}_Z,\mathcal{O}_{X_p})$ vanish. But then the exact sequence shows that we have $\iota_p^{*}\mathcal{I}_Z=\mathcal{O}_{X_p}$.
\end{proof}

\begin{lem}\label{seqid}
Let $p\in Y$ be such that $p\in f(Z)$, then there is an $n_p\in \mathbb{Z}_{>0}$ and an exact sequence
\begin{equation*}\begin{CD}
0 @>>> \mathcal{T}or_1^{\iota_p^{-1}\mathcal{O}_X}(\iota_p^{-1}\mathcal{O}_Z,\mathcal{O}_{X_p}) @>>> \iota_p^{*}\mathcal{I}_Z @>>> \mathcal{O}_{X_p}(-n_p) @>>> 0.
\end{CD}\end{equation*} 
\end{lem}
\begin{proof}
Since $p\in f(Z)$ is equivalent to $Z\cap X_p\neq \emptyset$ we know that $\mathcal{Q}:=\iota_p^{*}\mathcal{O}_Z$ is a nontrivial $\mathcal{O}_{X_p}$-module of finite length as it vanishes at the generic point of $X_p$. Define $n_p:=l_{\mathcal{O}_{X_p}}(\mathcal{Q})$, then we get the exact sequence:
\begin{equation*}\begin{CD}
0 @>>> \mathcal{O}_{X_p}(-n_p) @>>> \mathcal{O}_{X_p} @>>> \mathcal{Q} @>>> 0.
\end{CD}\end{equation*} 
Using the exactness of (\ref{pullb}) shows that we get the desired sequence.
\end{proof}

\subsection{Indecomposable locally free sheaves}
There is the following cohomological criterion to test if a locally free sheaf of rank two is decomposable:
\begin{lem}[\normalfont{\cite[Lemma 5.]{hanna}}]
A locally free sheaf $\mathcal{E}$ of rank two on $X$ is decomposable if and only if for all $n\in\mathbb{Z}$ the $\mathbb{Z}$-module $H^i(X,\mathcal{E}(n))$ is free $(i=0,1)$.
\end{lem}
We want to prove a criterion for decomposability of a locally free sheaf of rank two in terms of the associated local complete intersection. 
\begin{lem}\label{inde}
Let $\mathcal{E}$ be a locally free sheaf of rank two on $X$. If the local complete intersection in the associated triple $(a,b,Z)$ is not empty, then $\mathcal{E}$ is indecomposable.
\end{lem}
\begin{proof}
We will compute the locally free sheaf $\mathcal{E}_p$ on $X_p$ for all $p\in Y$.\\
Using (\ref{exseq1}) we get for $p\notin f(Z)$ the exact sequence
\begin{equation}\begin{CD}\label{seq2}
0 @>>> \mathcal{O}_{X_p}(a) @>>> \mathcal{E}_p @>>> \iota_p^{*}\mathcal{I}_Z(b) @>>> 0,
\end{CD}\end{equation}
as $\mathcal{T}or_1^{\iota_p^{-1}\mathcal{O}_X}(\iota_p^{-1}\mathcal{I}_Z(b),\mathcal{O}_{X_p})=0$ because it is a torsion sheaf and a subsheaf of the torsion free sheaf $\mathcal{O}_{X_p}(a)$.\\ 
We know by (\ref{idstr}) that $\iota_p^{*}\mathcal{I}_Z(b)=\mathcal{O}_{X_p}(b)$. So we get the exact sequence
\begin{equation*}\begin{CD}
0 @>>> \mathcal{O}_{X_p}(a) @>>> \mathcal{E}_p @>>> \mathcal{O}_{X_p}(b) @>>> 0.
\end{CD}\end{equation*}
But $Ext^1_{\mathcal{O}_{X_p}}(\mathcal{O}_{X_p}(b),\mathcal{O}_{X_p}(a))=0$, so that we have $\mathcal{E}_p\cong\mathcal{O}_{X_p}(a)\oplus\mathcal{O}_{X_p}(b)$.\\
Now pick $p\in f(Z)$. Then we start with the same sequence (\ref{seq2}). Using (\ref{seqid}) we get a surjection $\iota_p^{*}\mathcal{I}_Z(b)\twoheadrightarrow \mathcal{O}_{X_p}(b-n_p)$. This induces a surjection $\phi: \mathcal{E}_p\twoheadrightarrow \mathcal{O}_{X_p}(b-n_p)$. Define $\mathcal{L}:=ker(\phi)$, then $\mathcal{L}$ must be a line bundle since $X_p$ is a smooth curve and $\mathcal{E}_p$ is locally free. As the degree of $\mathcal{E}_p$ is constant on the fibers, see (\cite[Theorem 3.1.]{roberts}), we have $deg(\mathcal{E}_p)=a+b$. This implies $\mathcal{L}\cong\mathcal{O}_{X_p}(a+n_p)$. Thus we get the following exact sequence:
\begin{equation*}\begin{CD}
0 @>>> \mathcal{O}_{X_p}(a+n_p) @>>> \mathcal{E}_p @>>> \mathcal{O}_{X_p}(b-n_p) @>>> 0.
\end{CD}\end{equation*}
A short computation gives $Ext^1_{\mathcal{O}_{X_p}}(\mathcal{O}_{X_p}(b-n_p),\mathcal{O}_{X_p}(a+n_p))=0$, so that in this case we have $\mathcal{E}_p=\mathcal{O}_{X_p}(a+n_p)\oplus\mathcal{O}_{X_p}(b-n_p)$.\\
Now if $\mathcal{E}$ were decomposable, that is $\mathcal{E}=\mathcal{O}_X(r)\oplus\mathcal{O}_X(s)$ for $r,s\in\mathbb{Z}$, then the decomposition into the sum of two line bundles had to be the same on every fiber. But for $p\in f(Z)$ we have $n_p>0$, so on finitely many fibers we have different decompositions than the generic decomposition. This shows that $\mathcal{E}$ must be indecomposable.
\end{proof}
\begin{thm}
A locally free sheaf $\mathcal{E}$ of rank two on $X$ is indecomposable if and only if the associated local complete interesection $Z$ is not empty.
\end{thm}
\begin{proof}
$(\Rightarrow)$ Assume $Z=\emptyset$. Then there is an exact sequence
\begin{equation*}\begin{CD}
0 @>>> \mathcal{O}_X(a) @>>> \mathcal{E} @>>> \mathcal{O}_X(b) @>>> 0.
\end{CD}\end{equation*}
Using the knowledge of the cohomology groups on $X$ and the fact that $a-b\geq 0$ we see that $Ext^1_{\mathcal{O}_X}(\mathcal{O}_X(b),\mathcal{O}_X(a))=0$. So there is an isomorphism $\mathcal{E}\cong\mathcal{O}_X(a)\oplus\mathcal{O}_X(b)$ and the sheaf is decomposable.\\
$(\Leftarrow)$ This is the previous lemma (\ref{inde}).
\end{proof}
\begin{rem}
\normalfont The lemma and the theorem are arithmetic analogues of the following situation on $\mathbb{P}^2_{\mathbb{C}}$, see (\cite[Exercise 2.3.]{fried}) or (\cite[Example 5.4., p.728]{grif}): let $V$ be a locally free sheaf of rank two which sits in an exact sequence 
\begin{equation*}\begin{CD}
0 @>>> \mathcal{O}_{\mathbb{P}^2} @>>> V @>>> \mathcal{I}_Z @>>> 0,
\end{CD}\end{equation*}
then $V$ is not a direct sum of line bundles if $Z\neq \emptyset$. Here $Z$ is a local complete intersection of codimension two in $\mathbb{P}^2_{\mathbb{C}}$.  
\end{rem}
\begin{rem}
\normalfont Indecomposable locally free sheaves of rank two on $X$ have also been discussed in (\cite{roberts}). Roberts constructs these using transition matrices with respect to the standard affine open cover of $X$. These matrices, elements of $GL(2,\mathbb{Z}[t,t^{-1}])$, have the property that their reductions have different diagonalizations over different primes. More generally it is shown that for every positive integer $r$ there exists an indecomposable locally free sheaf of rank $r$ on $X$.
\end{rem}
\begin{exam}
\normalfont We choose a prime $p\in Y$. Furthermore we pick a closed point $x\in X_p$ such that $[k(x):\mathbb{F}_p]=1$. Define $Z:=\left\{x\right\}$. Then $Z$ is a local complete intersection of codimension two in $X$. So there exists a locally free sheaf $\mathcal{E}$ of rank two on $X$ sitting in the exact sequence
\begin{equation*}\begin{CD}
0 @>>> \mathcal{O}_X(-1) @>>> \mathcal{E} @>>> \mathcal{I}_Z(-1) @>>> 0.
\end{CD}\end{equation*}
The computations in (\ref{inde}) show that we have $\mathcal{E}_q\cong \mathcal{O}_{X_q}(-1)\oplus \mathcal{O}_{X_q}(-1)$ for $q\neq p$ and $\mathcal{E}_p\cong \mathcal{O}_{X_p}\oplus \mathcal{O}_{X_p}(-2)$.\\
In \cite{roberts} this example is described with the help of the transition matrix
 $A=\begin{pmatrix}
t^2 & pt \\
0 & 1
 \end{pmatrix}$.
It is obvious that the reduction modulo p is given by $A_p=\begin{pmatrix}
t^2 & 0 \\
0 & 1
 \end{pmatrix}$. But this matrix is exactly the transition matrix of $\mathcal{O}_{X_p}\oplus \mathcal{O}_{X_p}(-2)$ on $X_p$. But an easy computation, using elementary column transformations over $(\mathbb{Z}/q\mathbb{Z})[t]$ and elementary row transformations over $(\mathbb{Z}/q\mathbb{Z})[t^{-1}]$, shows that for $q\neq p$ we have $A_q\sim \begin{pmatrix}
t & 0 \\
0 & t 
 \end{pmatrix}$. This is the transition matrix of $\mathcal{O}_{X_q}(-1)\oplus \mathcal{O}_{X_q}(-1)$ on $X_q$.
\end{exam}

\subsection{Cohomology}
In this section we first recall known results about the cohomology groups of line bundles on $X$. We go on and study the cohomology of the twisted ideal sheaf $\mathcal{I}_Z(b)$ for a local complete intersection of codimension two in $X$. These results can then be used to study the cohomology of locally free sheaves of rank two on $X$. For example we describe all locally free sheaves of rank two with no cohomology.
\begin{lem}[\normalfont{\cite[Theorem III.5.1]{hart}}]\label{cohomline}
The  $\mathbb{Z}$-module $H^i(X,\mathcal{O}_X(a))$ is free and we have: 
\begin{center}
$rk(H^0(X,\mathcal{O}_X(a))) = \begin{cases} a+1 &\mbox{if } a\geq 0 \\
0 &\mbox{if } a\leq -1. \end{cases}$ 
\end{center}
Using Serre duality we also get:
\begin{center}
$rk(H^1(X,\mathcal{O}_X(a))) = \begin{cases} -a-1 &\mbox{if } a\leq -2 \\
0 &\mbox{if } a\geq -1. \end{cases}$ 
\end{center}
\end{lem}
\begin{rem}\label{excohom}
\normalfont In fact, we can explicitly describe these modules. For $a\geq 0$ we have: 
\begin{center}
$H^0(X,\mathcal{O}_X(a))=\bigoplus\limits_{i=0}^a\mathbb{Z}x_0^{a-i}x_1^i$. 
\end{center}
And there is a perfect pairing of $\mathbb{Z}$-modules: $H^0(X,\mathcal{O}_X(a))\times H^1(X,\mathcal{O}_X(-a-2)) \rightarrow \mathbb{Z}$.
\end{rem}
\begin{lem}
The  $\mathbb{Z}$-module $H^0(X,\mathcal{I}_Z(b))$ is free and we have: 
\begin{center}
$rk(H^0(X,\mathcal{I}_Z(b))) = \begin{cases} b+1 &\mbox{if } b\geq 0 \\
0 &\mbox{if } b\leq -1. \end{cases}$ 
\end{center}
\end{lem}
\begin{proof}
Using the previous lemma (\ref{cohomline}) and the exact sequence 
\begin{equation}\begin{CD}\label{idealseq}
0 @>>> \mathcal{I}_Z(b) @>>> \mathcal{O}_X(b) @>>> \mathcal{O}_Z @>>> 0
\end{CD}\end{equation}
we conclude that for $b\leq -1$ we must have $H^0(X,\mathcal{I}_Z(b))=0$. The sequence also shows that $H^0(X,\mathcal{I}_Z(b))$ is a submodule of a free $\mathbb{Z}$-module, hence it must be itself free. To find its rank we use the fact that $\mathbb{Q}$ is a flat $\mathbb{Z}$-module, hence (\cite[Proposition III.9.3]{hart}) gives an isomorphism 
\begin{center}
$H^0(X,\mathcal{I}_Z(b))\otimes_{\mathbb{Z}}\mathbb{Q}\cong H^0(X_0,\iota_0^{*}\mathcal{I}_Z(b))$. 
\end{center}
But by (\ref{idstr}) we know that $\iota_0^{*}\mathcal{I}_Z(b)=\mathcal{O}_{X_0}(b)$. This shows that for $b\geq 0$ the rank of $H^0(X,\mathcal{I}_Z(b))$ is $b+1$.
\end{proof}
\begin{lem}
The $\mathbb{Z}$-module $H^1(X,\mathcal{I}_Z(b))$ is torsion for $b\geq 0$. Furthermore if we define the number $n_p:=l_{\mathcal{O}_{X_p}}(\iota_p^{*}\mathcal{O}_Z)$ for $p\in Y$, then $H^1(X,\mathcal{I}_Z(b))=0$ if $b\geq \max\limits_{p\in Y}\left\{n_p-1\right\}$.
\end{lem}
\begin{proof}
For the first assertion we note that in this case $H^1(X,\mathcal{I}_Z(b))$ is a quotient of the torsion module $H^0(Z,\mathcal{O}_Z)$.\\
By (\ref{idstr}) we know that if $p\notin f(Z)$, then $\iota_p^{*}\mathcal{I}_Z(b)=\mathcal{O}_{X_p}(b)$ and as $b\geq -1$ we see that $H^1(X_p,\iota_p^{*}\mathcal{I}_Z(b))=0$. Now if $p\in f(Z)$ we use (\ref{seqid}) and see that there is an isomorphism 
\begin{center}
$H^1(X_p,\iota_p^{*}\mathcal{I}_Z(b))\cong H^1(X_p,\mathcal{O}_{X_p}(b-n_p))$. 
\end{center}
By the choice of $b$ the latter group vanishes. So we have $H^1(X_p,\iota_p^{*}\mathcal{I}_Z(b))=0$ for all $p\in Y$. As $\mathcal{I}_Z(b)$ is a torsion free $\mathcal{O}_X$-module, it is flat over $Y$. This implies we can use Grauert's base change theorem, (\cite[Corollary III.12.9]{hart}), so that 
\begin{center}
$H^1(X,\mathcal{I}_Z(b))\otimes k(p) \cong H^1(X_p,\iota_p^{*}\mathcal{I}_Z(b))=0$
\end{center}
for all $p\in Y$. But this is only possible if $H^1(X,\mathcal{I}_Z(b))=0$.
\end{proof}
\begin{lem}\label{vanish}
For $b\leq -1$ there is an isomorphism 
\begin{center}
$H^1(X,\mathcal{I}_Z(b))\cong H^0(Z,\mathcal{O}_Z)\oplus H^1(X,\mathcal{O}_X(b))$.
\end{center}
\end{lem}
\begin{proof}
The long exact sequence associated to (\ref{idealseq}) reduces to the short exact sequence
\begin{equation*}\begin{CD}
0 @>>> H^0(Z,\mathcal{O}_Z) @>>> H^1(X,\mathcal{I}_Z(b)) @>>>  H^1(X,\mathcal{O}_X(b)) @>>> 0.
\end{CD}\end{equation*}
As $H^1(X,\mathcal{O}_X(b))$ is free, this sequence splits.
\end{proof}

\begin{thm}
A locally free sheaf $\mathcal{E}$ of rank two on $X$ has no cohomology if and only if $\mathcal{E}\cong \mathcal{O}_X(-1)\oplus\mathcal{O}_X(-1)$.
\end{thm}
\begin{proof}
By (\ref{exseq1}) we can write $\mathcal{E}$ as an extension
\begin{equation*}\begin{CD}
0 @>>> \mathcal{O}_X(a) @>>> \mathcal{E} @>>> \mathcal{I}_Z(b) @>>> 0 
\end{CD}\end{equation*}
with $a\geq b$ and $Z$ a local complete intersection of codimension two in $X$. The long exact sequence in cohomology shows that 
\begin{enumerate}
\item $H^0(X,\mathcal{O}_X(a))=0$ 
\item $H^0(X,\mathcal{I}_Z(b))\cong H^1(X,\mathcal{O}_X(a))$
\item $H^1(X,\mathcal{I}_Z(b))=0$. 
\end{enumerate}
Now (1.) implies $a\leq -1$. So we must have $b\leq -1$. So $H^0(X,\mathcal{I}_Z(b))=0$. Thus by (2.) $H^1(X,\mathcal{O}_X(a))=0$. This is only possible if $a=-1$. So still $b\leq -1$, but then (\ref{vanish}) shows that for (3.) to be true we need $b=-1$ and $Z=\emptyset$. The exact sequence reads
\begin{equation*}\begin{CD}
0 @>>> \mathcal{O}_X(-1) @>>> \mathcal{E} @>>> \mathcal{O}_X(-1) @>>> 0.
\end{CD}\end{equation*}
But this sequence splits. So $\mathcal{E}\cong\mathcal{O}_X(-1)\oplus\mathcal{O}_X(-1)$. The other implication is clear.
\end{proof}

\section{Arakelov geometry}
In this section we will study locally free sheaves of rank two on $X$ in the Arakelov geometric setting. For this purpose we also need Hermitian coherent sheaves. These are pairs $(\mathcal{F},h)$ where $\mathcal{F}$ is a coherent sheaf on $X$ such that $\mathcal{F}_0$ is locally free and $h$ is a Hermitian metric on the associated vector bundle $F$ on $X(\mathbb{C})$. Such a pair will be denoted by $\overline{\mathcal{F}}$.\\ 
In the following we will equip the manifold $X(\mathbb{C})=\mathbb{C}\mathbb{P}^1$ with the Fubini-Study metric. This induces a metric, also called Fubini-Study metric, on $O_{\mathbb{C}\mathbb{P}^1}(1)$. Using tensor products and duals this defines the Fubini-Study metric on $O_{\mathbb{C}\mathbb{P}^1}(m)$ for any $m\in\mathbb{Z}$, thus we get a Hermitian vector bundle $\overline{\mathcal{O}_X(m)}$.\\ 
Now if $E$ is a rank 2 vector bundle on $\mathbb{C}\mathbb{P}^1$, then it is a direct sum of line bundles. Each line bundle is equipped with the Fubini-Study metric and we use the orthogonal direct sum metric on $E$ such that $(E,h)=(L_1\oplus L_2,h'\oplus h'')$.\\
If $Z$ is a local complete intersection of codimension two in $X$, then $\mathcal{I}_Z(b)$ is a coherent sheaf on $X$ and $\mathcal{I}_Z(b)_0=\mathcal{O}_{X_0}(b)$ is locally free. We equip the associated line bundle $O_{\mathbb{C}\mathbb{P}^1}(b)$ on $\mathbb{C}\mathbb{P}^1$ with the Fubini-Study metric, so that we get a Hermitian coherent sheaf $\overline{\mathcal{I}_Z(b)}$ on $X$.\\
For a locally free sheaf $\mathcal{E}$ of rank two on $X$ we have an exact sequence:
\begin{equation*}\begin{CD}
0 @>>> \mathcal{O}_X(a) @>>> \mathcal{E} @>>> \mathcal{I}_Z(b) @>>> 0.
\end{CD}\end{equation*}
We get the associated sequence on $\mathbb{C}\mathbb{P}^1$:
\begin{equation*}\begin{CD}
0 @>>> O_{\mathbb{C}\mathbb{P}^1}(a) @>>> E @>>> O_{\mathbb{C}\mathbb{P}^1}(b) @>>> 0, 
\end{CD}\end{equation*}
so that the associated bundle is $E=O_{\mathbb{C}\mathbb{P}^1}(a)\oplus O_{\mathbb{C}\mathbb{P}^1}(b)$ with the metric $h$ such that $(E,h)=(O_{\mathbb{C}\mathbb{P}^1}(a)\oplus O_{\mathbb{C}\mathbb{P}^1}(b),h'\oplus h'')$. Thus we have a Hermitian vector bundle $\overline{\mathcal{E}}$.

\subsection{Arithmetic Chern classes}
If $\mathcal{E}$ is a locally free sheaf of rank two on $X$, then there is an associated triple $(a,b,Z)$ with $a,b\in \mathbb{Z}$ such that $a\geq b$ and $Z$ a local complete intersection of codimension two in $X$ and $\mathcal{E}$ is an extension of $\mathcal{O}_X(a)$ by $\mathcal{I}_Z(b)$. So we should really write $\mathcal{E}_Z^{(a,b)}$ for the sheaf. In the following if we have a locally free sheaf of rank two we always write $\mathcal{E}$ with the associated triple $(a,b,Z)$ in mind.\\
Since any locally free sheaf of rank two induces a Hermitian vector bundle $\overline{\mathcal{E}}$, we can compute its arithmetic Chern classes. As $\mathcal{E}$ sits in the exact sequence
\begin{equation*}\begin{CD}
0 @>>> \mathcal{O}_X(a) @>>> \mathcal{E} @>>> \mathcal{I}_Z(b) @>>> 0,
\end{CD}\end{equation*}
we obviously want to find the Chern classes of $\mathcal{E}$ using the Chern classes of $\mathcal{O}_X(a)$ and $\mathcal{I}_Z(b)$. But the arithmetic Chern classes $\widehat{c}_i$ do not behave well on exact sequences. The so called Bott-Chern secondary classes will usually appear. But fortunately in our situation the vector bundles on $\mathbb{C}\mathbb{P}^1$ have a simple form and by our choice of the metric $h$, the exact sequence is split, see (\cite[1.2.1.]{soul}) for a definition of split. But then all Bott-Chern secondary classes vanish by (\cite[Theorem 1.2.2.(iii)]{soul}).
\begin{lem}\label{ceins}
Let $\mathcal{I}_Z(b)$ be the twisted ideal sheaf of a local complete intersection of codimension two in $X$, then
\begin{center}
$\widehat{c}_1(\overline{\mathcal{I}_Z(b)})=\widehat{c}_1(\overline{\mathcal{O}_X(b)})$.
\end{center}
\end{lem}
\begin{proof}
Given a torsion free coherent sheaf $\mathcal{F}$ on $X$, we can look at the double dual $\mathcal{F}^{**}$. Since $\mathcal{F}$ is torsion free it only differs from its double dual in codimension two, hence $\mathcal{F}$ and $\mathcal{F}^{**}$ induce the same vector bundle on $X(\mathbb{C})$ which we will equip with a Hermitian metric $h$. By (\cite[Corollary 8.9.]{mori1}) we have the equality $\widehat{c}_1(\mathcal{F},h)=\widehat{c}_1(\mathcal{F}^{**},h)$. Applying this to $\mathcal{F}=\mathcal{I}_Z(b)$ and remembering that we chose the Fubini-Study metric on the vector bundle associated to $\mathcal{I}_Z(b)$ we get the desired result, since $\mathcal{I}_Z(b)^{**}=\mathcal{O}_X(b)$.
\end{proof}
\begin{rem}
\normalfont In fact, the arithmetic first Chern class defines an isomorphism 
\begin{center}
$\widehat{c}_1: \widehat{Pic}(W)\stackrel{\sim}{\longrightarrow} \widehat{CH}$$^1(W)$ 
\end{center}
for any regular arithmetic variety $W$. The map is given by 
\begin{center}
$\widehat{c}_1(\overline{\mathcal{L}})=(div(s),-log(h(s,s)))$ 
\end{center}
for a non-zero rational section $s$ of $\mathcal{L}$. See (\cite[Proposition III.4.2.1]{soul3}) for a proof of this fact.
\end{rem}

\begin{lem}\label{czwei}
Let $\mathcal{I}_Z(b)$ be the twisted ideal sheaf of a local complete intersection of codimension two in $X$, then
\begin{center}
$\widehat{c}_2(\overline{\mathcal{I}_Z(b)})=log(\#H^0(Z,\mathcal{O}_Z))$.
\end{center}
\end{lem}
\begin{proof}
It is known that if $\mathcal{F}$ is a torsion free coherent sheaf on $X$, then one has 
\begin{center}
$\widehat{c}_2(\mathcal{F},h)=\widehat{c}_2(\mathcal{F}^{**},h)+log(\#(\mathcal{F}^{**}/\mathcal{F}))$, 
\end{center}
see (\cite[Corollary 8.9.]{mori1}).\\ 
We apply this to $\mathcal{F}=\mathcal{I}_Z(b)$. Since $\mathcal{O}_X(b)$ has rank one $\widehat{c}_2(\mathcal{O}_X(b))=0$, see (\cite[4.9.]{soul}). Furthermore $\#(\mathcal{O}_X(b)/\mathcal{I}_Z(b))=\#H^0(Z,\mathcal{O}_Z)$ as $Z$ has codimension two in $X$. So we get the result. 
\end{proof}
\begin{rem}
\normalfont Here we use the ususal notation, which says we can identify an element $x\in \widehat{CH}$$^2(X)$ with its degree $\widehat{deg}(x)\in\mathbb{R}$, see (\cite[Section 2]{soul4}). So $x$ may denote the element in the second Chow group or its degree. It will be clear from the context what meaning it has in the case at hand. For the sake of completeness we give the definition of the degree function:\\ 
if $x\in \widehat{CH}$$^2(X)$, then $x=(Z,g)$ and we have 
\begin{center}
$\widehat{deg}(x)=log(\#(H^0(Z,\mathcal{O}_Z)))+\frac{1}{2}\displaystyle\int_{X(\mathbb{C})} g$. 
\end{center}
So we could also write $\widehat{c}_2(\overline{\mathcal{I}_Z(b)})=Z$ for the second arithmetic Chern class.
\end{rem}

\begin{thm}\label{chern}
Let $\mathcal{E}$ be a locally free sheaf of rank two on $X$. Then 
\begin{align*}
\widehat{c}_1(\overline{\mathcal{E}})&=\widehat{c}_1(\overline{\mathcal{O}_X(a+b)})\\
\widehat{c}_2(\overline{\mathcal{E}})&=\widehat{c}_1(\overline{\mathcal{O}_X(a)})\widehat{c}_1(\overline{\mathcal{O}_X(b)})+log(\#H^0(Z,\mathcal{O}_Z)).
\end{align*}
\end{thm}
\begin{proof}
Since the exact sequence (\ref{exseq1}) is a locally free resolution of $\mathcal{I}_Z(b)$, we get 
\begin{center}
$\widehat{ch}(\overline{\mathcal{I}_Z(b)})=\widehat{ch}(\overline{\mathcal{E}})-\widehat{ch}(\overline{\mathcal{O}_X(a)})$
\end{center}
by (\cite[Corollary 3.1.4]{soul2}). Denote by $\widehat{ch}_i(\mathcal{E})$ the $\widehat{CH}$$^{i}(X)_{\mathbb{Q}}$-part of $\widehat{ch}(\mathcal{E})$. So for $i=1$ we get $\widehat{ch}_1(\mathcal{E})=\widehat{c}_1(\mathcal{E})$. Putting everything in we see 
\begin{center}
$\widehat{c}_1(\overline{\mathcal{E}})=\widehat{c}_1(\overline{\mathcal{O}_X(a)})+\widehat{c}_1(\overline{\mathcal{I}_Z(b)})$. 
\end{center}
Using lemma (\ref{ceins}) and (\cite[4.1.4]{soul}) we get $\widehat{c}_1(\overline{\mathcal{E}})=\widehat{c}_1(\overline{\mathcal{O}_X(a+b)})$.\\
Furthermore $\widehat{ch}_2(\overline{\mathcal{E}})$ is given by $\widehat{ch}_2(\overline{\mathcal{E}})=\frac{1}{2}\widehat{c}_1(\overline{\mathcal{E}})^2-\widehat{c}_2(\overline{\mathcal{E}})$ and we get 
\begin{center}
$\widehat{ch}_2(\overline{\mathcal{I}_Z(b)})=\frac{1}{2}\widehat{c}_1(\overline{\mathcal{O}_X(b)})^2-log(\#H^0(Z,\mathcal{O}_Z))$, 
\end{center}
using lemmas (\ref{ceins}) and (\ref{czwei}).\\
We also have $\widehat{ch}_2(\overline{\mathcal{O}_X(a)})=\frac{1}{2}\widehat{c}_1(\overline{\mathcal{O}_X(a)})^2$ since $\widehat{c}_2(\overline{\mathcal{O}_X(a)})=0$ by \cite[4.9]{soul}. Putting everything in, a short computation gives the result.
\end{proof}
\begin{rem}
\normalfont A first nontrivial example, this means that the exact sequence of associated vector bundles is not split, is given as follows: choose a Hermitian metric $h$ on $E$ that is not the orthogonal direct sum metric. Equip the the line bundles $L_1$ and $L_2$ with the submetric induced by h, respectively the quotient metric induced by $h$ by regarding $L_2$ as the orthogonal complement of $L_1$. Then $\widehat{c}_1$ is still additive on exact sequences, see (\cite[1.2.5.]{soul}). But in this case the secondary Bott-Chern class for $\widehat{c}_2$ does not vanish in general and one can write down a formula, see for example (\cite[Lemma 2.1.1.]{soul4}) or (\cite[Theorem 5.1.]{gasb}).
\end{rem} 
Using the arithmetic Chern classes, we can define the arithmetic discriminant $\widehat{\Delta}(\overline{\mathcal{E}})$ of a locally free sheaf $\mathcal{E}$ of rank two on $X$. If the sheaf has Chern classes $\widehat{c}_1$ and $\widehat{c}_2$, then the discriminat is defined by $\widehat{\Delta}=4\widehat{c}_2-\widehat{c}_1^2$.
\begin{lem}
Let $\mathcal{E}$ be a locally free sheaf of rank two on $X$, then 
\begin{center}
$\widehat{\Delta}(\overline{\mathcal{E}})=4log(\#H^0(Z,\mathcal{O}_Z))-\frac{1}{2}(a-b)^2$.
\end{center}
\end{lem}
\begin{rem}
\normalfont Using the arithmetic discriminant one can study arithmetic stability respectively unstability, see (\cite{mori3}). Furthermore one can prove an arithmetic analogue of the Bogomolov-Gieseker inequality for semistable vector bundles on arithmetuic surfaces, see (\cite{mori1}).
\end{rem}

\subsection{Arithmetic Riemann-Roch}
In this section we want to recall the arithmetic Riemann-Roch theorem. For this purpose let $g: W\rightarrow Z$ be a smooth projective morphism of arithmetic varieties over $\mathbb{Z}$. Choose a Hermitian metric $h_{W/Z}$ on the relative tangent sheaf $T_{W/Z}$ which induces a K\"ahler metric on each fiber $g^{-1}(z)$ for $z\in Z(\mathbb{C})$. This defines a Hermitian vector bundle $\overline{T_{W/Z}}$. Now given a Hermitian vector bundle $\overline{\mathcal{E}}=(\mathcal{E},h)$ on $W$ we can define a line bundle $\lambda(\mathcal{E})$ on $Z$ by 
\begin{center}
$\lambda(\mathcal{E}):=\bigotimes\limits_{i\geq 0}det(R^ig_{*}\mathcal{E})^{(-1)^i}$. 
\end{center}
We equip $\lambda(\mathcal{E})$ with the Quillen metric $h_Q$ and get a Hermitian line bundle $\overline{\lambda(\mathcal{E})}$ on $Z$. For more information about $\lambda(\mathcal{E})$ and the Quillen metric see (\cite[VI]{soul3}).\\
Then the arithmetic Riemann-Roch has the following form:
\begin{thm}\label{arr}
There is the following equality in $\widehat{CH}$$^1(Z)$:
\begin{center}
$\widehat{c}_1(\overline{\lambda(\mathcal{E})})=g_{*}(\widehat{ch}(\overline{\mathcal{E}})\widehat{Td}$$^A(\overline{T_{W/Z}}))^{(1)}$
\end{center}
\end{thm}
Here given $\alpha\in \widehat{CH}(Z)_\mathbb{Q}$ we denote by $\alpha^{(1)}$ its component in $\widehat{CH}$$^{1}(Z)\otimes_{\mathbb{Z}}\mathbb{Q}$. In addition $\widehat{Td}$$^A$ is the so called arithmetic Todd genus, it is defined for a Hermitian vector bundle $\overline{\mathcal{E}}$ on $W$ by 
\begin{center}
$\widehat{Td}$$^A(\overline{\mathcal{E}})=\widehat{Td}(\overline{\mathcal{E}})(1-a(R(E)))$
\end{center}
in $\widehat{CH}(W)_\mathbb{Q}$. The class $R(E)$ is a new characteristic class for the vector bundle $E$ on $W(\mathbb{C})$. This class lives in the even complex cohomology part, that is $R(E)\in H^{ev}(W(\mathbb{C}))$. If $E=L$ is a line bundle, one has the following formula:
\begin{center}
$R(L)=\sum\limits_{\genfrac{}{}{0pt}{}{k\hspace{0.1cm}\text{odd}}{k\geq 1}} (2\zeta'(-k)+\zeta(-k)(1+\frac{1}{2}+\cdots+\frac{1}{k}))\frac{c_1(L)^k}{k!}$
\end{center}
\begin{rem}
\normalfont The arithmetic Riemann-Roch theorem (\ref{arr}) was conjectured in (\cite{soul5}). For a proof of this theorem (even in a far more general setting) see (\cite[Theorem 4.1.4.7]{soul2}). For more information about the arithmetic Chow ring $\widehat{CH}(W)$, the pushforward map $g_{*}$ and the map $a$ we refer to (\cite[III]{soul3}). 
\end{rem}
If $\overline{L}=(L,h)$ is a Hermitian line bundle on $Y=Spec(\mathbb{Z})$, we define its Arakelov-degree by: 
\begin{center}
$\widehat{deg}(\overline{L}):=log(\#(L/s\mathbb{Z}))-\frac{1}{2}log(h(s,s))$
\end{center}
where $s$ is a nontrivial global section of $L$. This defines a morphism $\widehat{Pic}(Y)\rightarrow \mathbb{R}$.\\
If $\overline{\mathcal{F}}$ is a Hermitian coherent sheaf on $Y$, that is a finitely generated $\mathbb{Z}$-module with a metric on $\mathcal{F}\otimes_{\mathbb{Z}}\mathbb{C}$, we define $\widehat{deg}(\overline{\mathcal{F}}):=\widehat{deg}(\overline{det(\mathcal{F})})$. One can easily show that this degree is given by
\begin{equation}\label{degree}
\widehat{deg}(\overline{\mathcal{F}})=log(\#(\mathcal{F}/(s_1\mathbb{Z}+\cdots+s_n\mathbb{Z})))-\frac{1}{2}log(det(h(s_i,s_j)))
\end{equation}
with $s_1,\ldots,s_n\in \mathcal{F}$ such that these elements form a basis of $\mathcal{F}\otimes_{\mathbb{Z}}\mathbb{Q}$.\\ 
There is also a description of the degree in terms of torsion and volume:
\begin{equation}\label{degvol}
\widehat{deg}(\overline{\mathcal{F}})=\log(\#\mathcal{F}_{tors})-log(vol(\mathcal{F}_{\mathbb{C}}^{+}/\mathcal{F})).
\end{equation}
Here the volume is the one induced by the Hermitian metric, $(\mathcal{F}_{\mathbb{C}})^{+}$ denotes the subspace fixed by the complex conjugation and the volume is measured of its quotient by the lattice $\mathcal{F}/\mathcal{F}_{tors}$, see for example (\cite[Lemma VIII.1]{soul3}) or (\cite[4.1.5.]{soul2}).\\
We also have a degree morphism $\widehat{deg}: \widehat{CH}^1(Y)\rightarrow \mathbb{R}$ given by 
\begin{center}
$\widehat{deg}(\sum n_ip_i,\mu):=\sum n_ilog(p_i)+\mu/2$.
\end{center}
A short computation shows that the different degree functions are compatible with the arithmetic first Chern class on $Y=Spec(\mathbb{Z})$: if $\overline{L}$ is an arithmetic line bundle on $Y$, then we have
\begin{equation}\label{degcone}
\widehat{deg}(\widehat{c}_1(\overline{L}))=\widehat{deg}(\overline{L}).
\end{equation}

\subsubsection{The arithmetic Hirzebruch-Riemann-Roch}
If we apply $\widehat{deg}$ to (\ref{arr}) in our situation $g=f$, $W=X$ and $Z=Y$ we get an arithmetic analogue of the Hirzebruch-Riemann-Roch theorem:\\
Using (\ref{degcone}) and the definition of the Hermitian line bundle $\overline{\lambda(\mathcal{E})}$ the left hand side gives:
\begin{center}
$\widehat{deg}(\widehat{c}_1(\overline{\lambda(\mathcal{E})}))=\widehat{deg}(H^0(X,\mathcal{E}),h_{L^2})-\widehat{deg}(H^1(X,\mathcal{E}),h_{L^2})+\frac{1}{2}T(E)$.
\end{center}
Here $T(E)$ is the Ray-Singer analytic torsion of the Hermitian vector bundle $E$ on $X(\mathbb{C})$.\\
We can define two arithmetic Euler characteristics:
\begin{itemize}
\item The $Q$-Euler characteristic $\chi_Q(\overline{\mathcal{E}}):=\widehat{deg}(\widehat{c}_1(\overline{\lambda(\mathcal{E})}))$
\item The $L^2$-Euler characteristic $\chi_{L^2}(\overline{\mathcal{E}})=\widehat{deg}(H^0(X,\mathcal{E}),h_{L^2})-\widehat{deg}(H^1(X,\mathcal{E}),h_{L^2})$.
\end{itemize}
They are obviously related by $\chi_Q(\overline{\mathcal{E}})=\chi_{L^2}(\overline{\mathcal{E}})+\frac{1}{2}T(E)$.\\
Now we can state the arithmetic Hirzebruch-Riemann-Roch theorem:
\begin{cor}
If $\overline{\mathcal{E}}$ is a Hermitian vector bundle on $X$, then we have:
\begin{center}
$\chi_Q(\overline{\mathcal{E}})=\widehat{deg}(f_{*}(\widehat{ch}(\overline{\mathcal{E}})\widehat{Td}$$^A(\overline{T_{X/Y}}))^{(1)})$
\end{center}
\end{cor}
Next we want to analyze the term on the right hand side:\\ 
Using the definition of $\widehat{Td}$$^A$ and the properties of $a$ we get
\begin{center}
$\widehat{ch}(\overline{\mathcal{E}})\widehat{Td}$$^A(\overline{T_{X/Y}})=\widehat{ch}(\overline{\mathcal{E}})\widehat{Td}(\overline{T_{X/Y}})-a(ch(E)R(T_{X(\mathbb{C})})Td(T_{X(\mathbb{C})}))$.
\end{center}
Now applying $f_{*}$ and looking at the component in $\widehat{CH}$$^1(Y)_{\mathbb{Q}}$ we see that the first term is given by
\begin{center}
$f_{*}(\widehat{ch}(\overline{\mathcal{E}})\widehat{Td}(\overline{T_{X/Y}}))^{(1)}=f_{*}(\frac{1}{2}\widehat{c}_1(\overline{\mathcal{E}})^2-\widehat{c}_2(\overline{\mathcal{E}})+\frac{1}{2}\widehat{c}_1(\overline{\mathcal{E}})\widehat{c}_1(\overline{T_{X/Y}})+\frac{rk(\mathcal{E})}{12}\widehat{c}_1(\overline{T_{X/Y}})^2)$.
\end{center}
Since there is an isometry $\overline{T_{X/Y}}\cong\overline{\mathcal{O}_X(2)}$, the term simplifies to:
\begin{center}
$f_{*}(\widehat{ch}(\overline{\mathcal{E}})\widehat{Td}(\overline{T_{X/Y}}))^{(1)}=f_{*}(\frac{1}{2}\widehat{c}_1(\overline{\mathcal{E}})^2-\widehat{c}_2(\overline{\mathcal{E}})+\widehat{c}_1(\overline{\mathcal{E}})\widehat{c}_1(\overline{\mathcal{O}_X(1)})+\frac{rk(\mathcal{E})}{3}\widehat{c}_1(\overline{\mathcal{O}_X(1)})^2)$.
\end{center}
But it is well known that we have 
\begin{equation}\label{oxtwo}
f_{*}(\widehat{c}_1(\overline{\mathcal{O}_X(1)})^2)=\frac{1}{2}, 
\end{equation}
see for example (\cite[4.1.7.(23)]{soul2}). So we finally have
\begin{center}
$f_{*}(\widehat{ch}(\overline{\mathcal{E}})\widehat{Td}(\overline{T_{X/Y}}))^{(1)}=f_{*}(\frac{1}{2}\widehat{c}_1(\overline{\mathcal{E}})^2-\widehat{c}_2(\overline{\mathcal{E}})+\widehat{c}_1(\overline{\mathcal{E}})\widehat{c}_1(\overline{\mathcal{O}_X(1)}))+\frac{rk(\mathcal{E})}{6}$.
\end{center}
On the other hand, for the second term we get after applying $\widehat{deg}\circ f_{*}$:
\begin{center}
$\widehat{deg}(f_{*}(a(ch(E)R(T_{X(\mathbb{C})})Td(T_{X(\mathbb{C})})))^{(1)})=\frac{1}{2}\displaystyle\int_{X(\mathbb{C})} (ch(E)R(T_{X(\mathbb{C})})Td(T_{X(\mathbb{C})}))^{(1)}$
\end{center}
see for example (\cite[Theorem VIII.2.2']{soul3}). With the simplified notation this reads
\begin{center}
$f_{*}(a(ch(E)R(T_{X(\mathbb{C})})Td(T_{X(\mathbb{C})})))^{(1)}=\frac{1}{2}\displaystyle\int_{X(\mathbb{C})} (ch(E)R(T_{X(\mathbb{C})})Td(T_{X(\mathbb{C})}))^{(1)}$
\end{center}
Remembering that the integral of the top Chern class gives us the Euler characteristic, that is here
\begin{center}
$\displaystyle\int_{X(\mathbb{C})} c_1(T_{X(\mathbb{C})})=\chi(X(\mathbb{C}))=2$,
\end{center}
we end up with
\begin{center}
$f_{*}(a(ch(E)R(T_{X(\mathbb{C})})Td(T_{X(\mathbb{C})})))^{(1)}=rk(E)(2\zeta'(-1)+\zeta(-1))$.
\end{center}
Summing up and using $\zeta(-1)=-\frac{1}{12}$, the right hand side is given by
\begin{center}
$f_{*}(\frac{1}{2}\widehat{c}_1(\overline{\mathcal{E}})^2-\widehat{c}_2(\overline{\mathcal{E}})+\widehat{c}_1(\overline{\mathcal{E}})\widehat{c}_1(\overline{\mathcal{O}_X(1)}))+rk(\mathcal{E})(\frac{1}{4}-2\zeta'(-1))$
\end{center}
These computations and simplifications give us the following easy formula for the arithmetic Hirzebruc-Riemann-Roch theorem:
\begin{cor}\label{ahrr}
Assume $\overline{\mathcal{E}}$ is a Hermitian vector bundle on $X$, then we have:
\begin{center}
$\chi_Q(\overline{\mathcal{E}})=f_{*}(\frac{1}{2}\widehat{c}_1(\overline{\mathcal{E}})^2-\widehat{c}_2(\overline{\mathcal{E}})+\widehat{c}_1(\overline{\mathcal{E}})\widehat{c}_1(\overline{\mathcal{O}_X(1)}))+rk(\mathcal{E})(\frac{1}{4}-2\zeta'(-1))$
\end{center}
\end{cor}

\subsection{Analytic torsion of line bundles}
In this section we want to use the arithmetic Hirzebruch-Riemann-Roch theorem to compute the Ray-Singer analytic torsion of the Hermitian line bundle $O_{\mathbb{C}\mathbb{P}^1}(a)$ on $X(\mathbb{C})=\mathbb{C}\mathbb{P}^1$ for $a\in \mathbb{Z}$. Since we know the arithmetic Chern classes of $\overline{\mathcal{O}_X(a)}$ and $\overline{T_{X/Y}}\cong\overline{\mathcal{O}_X(2)}$, we only need to compute the Arakelov-degree of the Hermitian $\mathbb{Z}$-module $H^i(X,\mathcal{O}_X(a))$ for $i=0,1$. By (\ref{cohomline}) we know these modules and so we get the following lemma:
\begin{lem}
Assume $a\geq 0$ and equip the $\mathbb{Z}$-module $H^0(X,\mathcal{O}_X(a))$ with the $L^2$-metric induced by the Fubini-Study metrics on the manifold $X(\mathbb{C})$ and on the Hermitian line bunlde $\overline{\mathcal{O}_X(a)}$. Then we have:
\begin{center}
$\widehat{deg}(H^0(X,\mathcal{O}_X(a)),h_{L^2})=-\frac{1}{2}(log(\prod\limits_{i=0}^a \frac{i!(a-i)!}{(a+1)!}))$.
\end{center}  
\end{lem}
\begin{proof}
Following formula (\ref{degree}) we pick $s_i=x_0^{a-i}x_1^i$ for $i=0,\ldots,a$. Then we immediately see 
\begin{center}
$log(\#(H^0(X,\mathcal{O}_X(a))/(s_0\mathbb{Z}+\cdots+s_a\mathbb{Z})))=0$. 
\end{center}
It remains to compute $det(h_{L^2}(s_i,s_j))$. The $L^2$-metric is given by 
\begin{center}
$h_{L^2}(s_i,s_j)=\displaystyle\int_{\mathbb{C}\mathbb{P}^1} h(s_i,s_j)\hspace{0.1cm}\omega$,
\end{center}
where $h$ is the Fubini Sutdy metric on $O_{\mathbb{C}\mathbb{P}^1}(a)$ and $\omega$ is the K\"ahler form associated to the Fubini-Study metric on $\mathbb{C}\mathbb{P}^1$, see (\cite[VI.3.1]{soul3}).\\ 
Now in the usual chart $(x_0\neq 0)$ with the local coordinate $z=\frac{x_1}{x_0}$ we get $s_i=z^i$, $\omega=\dfrac{\mathrm{i}}{2\pi}\dfrac{1}{(1+|z|^2)^2} dz\wedge d\overline{z}$ and $h(s_i,s_j)=\dfrac{z^i\overline{z}^j}{(1+|z|^2)^a}$. So we have to compute the integral 
\begin{center}
$\dfrac{\mathrm{i}}{2\pi}\displaystyle\int_{\mathbb{C}} \frac{z^i\overline{z}^j}{(1+|z|^2)^{a+2}} dz\wedge d\overline{z}$. 
\end{center}
Using polar coordinates for $\mathbb{C}$, we get $\dfrac{1}{\pi}\displaystyle\int_{0}^{\infty}\displaystyle\int_{0}^{2\pi} \frac{r^{i+j+1}\exp{\mathrm{i}\varphi(i-j)}}{(1+r^2)^{a+2}} drd\varphi$.\\ 
We immediately see
\begin{center}
$h_{L^2}(s_i,s_j) = \begin{cases} 2\displaystyle\int_{0}^{\infty} \frac{r^{2i+1}}{(1+r^2)^{a+2}} dr &\mbox{if } i=j\\
0 &\mbox{if } i\neq j . \end{cases}$ 
\end{center}
With the substitution $r^2=t$ we arrive at 
\begin{center}
$\displaystyle\int_{0}^{\infty} \frac{t^{i}}{(1+t)^{a+2}} dt = ((a+1)\binom{a}{i})^{-1}=\frac{i!(a-i)!}{(a+1)!}$. 
\end{center}
This implies $det(h_{L^2}(s_i,s_j))=\prod\limits_{i=0}^a \frac{i!(a-i)!}{(a+1)!}$.
\end{proof}
The right hand side of the arithmetic Hirzebruch-Riemann-Roch formula (\ref{ahrr}) can be computed using (\ref{oxtwo}). The result is:
\begin{center}
$\frac{1}{4}(a+1)^2-2\zeta'(-1)$. 
\end{center}
Putting everything together gives the mentioned result:
\begin{thm}
Assume $\mathbb{C}\mathbb{P}^1$ is equipped with the Fubini-Study metric. If the Hermitian line bundle $O_{\mathbb{C}\mathbb{P}^1}(a)$ has the induced Fubini-Study metric than we have for $a\geq 0$:
\begin{center}
$T(\mathbb{C}\mathbb{P}^1,O_{\mathbb{C}\mathbb{P}^1}(a))=\frac{1}{2}(a+1)^2+log(\prod\limits_{i=0}^a \frac{i!(a-i)!}{(a+1)!})-4\zeta'(-1)$
\end{center}
\end{thm}
\begin{rem}
\normalfont Denote the dual $\mathbb{Z}$-module $Hom_{\mathbb{Z}}(M,\mathbb{Z})$ by $M^{\vee}$. For $a\geq 1$ there is a canonical isomorphism 
\begin{center}
$H^1(X,\mathcal{O}_X(-a))\cong (H^0(X,\mathcal{O}_X(a-2)))^{\vee}$ (Serre duality).
\end{center} 
This isomorphism is compatible with the $L^2$-metric, see (\cite[Theorem 1.4.(ii), p.27]{soul5}). So we have an isometry 
\begin{center}
$(H^1(X,\mathcal{O}_X(-a)),h_{L^2})\cong ((H^0(X,\mathcal{O}_X(a-2)))^{\vee},h_{L^2}^{\vee})$.
\end{center}
Moreover there is a canonical isomorphism $det(M^{\vee})\cong det(M)^{\vee}$.\\ 
Using these isomorphisms and $\widehat{deg}(L^{\vee},h^{\vee})=-\widehat{deg}(L,h)$ shows that 
\begin{center}
$\widehat{deg}(H^1(X,\mathcal{O}_X(-a)),h_{L^2})=-\widehat{deg}(H^0(X,\mathcal{O}_X(a-2)),h_{L^2})=\frac{1}{2}log(\prod\limits_{i=0}^{a-2} \frac{i!(a-2-i)!}{(a-1)!})$. 
\end{center}
So we get for $a\geq 1$:
\begin{center}
$T(\mathbb{C}\mathbb{P}^1,O_{\mathbb{C}\mathbb{P}^1}(-a))=\frac{1}{2}(-a+1)^2+log(\prod\limits_{i=0}^{a-2} \frac{i!(a-2-i)!}{(a-1)!})-4\zeta'(-1)$.
\end{center}
\end{rem}
\begin{rem}
\normalfont In (\cite{weng}) Weng computes the Ray-Singer analytic torsion for line bundles on $\mathbb{C}\mathbb{P}^n$ both with Fubini-Study metric. His main theorem gives a formula for the torsion with the help of combinatorial equations. In (\cite[5.1]{weng}) he applies this to the case $n=1$. His result differs from ours by multiplcation with $(-1)$. It also conflicts with his main theorem, compare for example the sign of the $log$-term. Maybe this a simple miscalculation or it is due to the fact that there are two definitions of the Ray-Singer analytic torsion, also differing by multiplcation with $(-1)$. The one we are using here 
\begin{center}
$T(E)=\sum\limits_{q\geq 0} (-1)^q q\zeta_q'(0)$ 
\end{center}
used in (\cite{mori1}) for example. The other definition  is used for example in (\cite{soul2}) and is given by 
\begin{center}
$T(E)=\sum\limits_{q\geq 0} (-1)^{q+1} q\zeta_q'(0)$. 
\end{center}
We also note that usually the analytic torsion is defined via $exp(T(E))$.
\end{rem}

\subsection{Applications to locally free sheaves of rank 2}
Now we want to apply the arithmetic Hirzebruch-Riemann-Roch theorem to locally free sheaves of rank two on $X$.\\
Let $\mathcal{E}$ be a locally free sheaf of rank two with associated triple $(a,b,Z)$, then we know its arithmetic Chern classes by (\ref{chern}). So we can put them in (\ref{ahrr}) and get 
\begin{center}
$\chi_Q(\overline{\mathcal{E}})=\frac{1}{4}(a+b)^2-\frac{1}{2}ab+\frac{1}{2}(a+b)-log(\#(H^0(Z,\mathcal{O}_Z)))+\frac{1}{2}-4\zeta'(-1)$,
\end{center}
using (\ref{oxtwo}). This gives the following
\begin{lem}
Let $\mathcal{E}$ be a locally free sheaf of rank two on $X$, then we have:
\begin{center}
$\chi_Q(\overline{\mathcal{E}})=\frac{(a+1)^2}{4}+\frac{(b+1)^2}{4}-log(\#H^0(Z,\mathcal{O}_Z))-4\zeta'(-1)$.
\end{center}
\end{lem}
\begin{rem}
\normalfont One has numerically $\zeta'(-1)\approx -0,1654211$. More exactly the value is given by $\zeta'(-1)=\frac{1}{12}-log(A)$. Here $A$ is the so called Glaisher-Kinkelin constant. This constant is defined by 
\begin{center}
$A:=\lim\limits_{n\to \infty} \dfrac{K(n+1)}{n^{n^2/2+n/2+1/12}exp(-n^2/4)}$.
\end{center}
Here $K(n)$ is the $K$-function defined by $K(n):=\prod\limits_{i=0}^{n-1} i^i$.
\end{rem}
One possible application of this lemma is to compute the torsion of $H^1(X,\mathcal{E}(n))$ for all $n\in \mathbb{Z}$ if $\mathcal{E}$ is indecomposable. We need the following fact:
\begin{lem}
Let $\mathcal{E}$ be a locally free sheaf on $X$, then $H^0(X,\mathcal{E})_{tors}=0$.
\end{lem}
\begin{proof}
Since $X$ is integral and $\mathcal{E}$ is torsion free we have an injection 
\begin{center}
$H^0(X,\mathcal{E})\hookrightarrow \mathcal{E}_{\eta}$, 
\end{center}
here $\eta\in X$ is the generic point. Since $\mathcal{E}_{\eta}$ is a torsion free $\mathbb{Z}$-module we get the desired result.
\end{proof}
We know that for at least one $n\in \mathbb{Z}$ there has to be torsion, because otherwise $H^1(X,\mathcal{E}(n))$ is always free. But then $\mathcal{E}$ is decomposable by the result of Hanna.\\
The lemma and (\ref{degvol}) show 
\begin{center}
$\widehat{deg}((H^0(X,\mathcal{E}(n)),h_{L^2})= log(vol(H^0(X,\mathcal{E}(n))_{\mathbb{C}}^{+}/H^0(X,\mathcal{E}(n))))$. 
\end{center}
So if we knew the volume of $H^i(X,\mathcal{E}(n))$ for $i=0,1$, we could use the arithmetic Hirzebruch-Riemann-Roch theorem (\ref{ahrr}) to determine the torsion of $H^1(X,\mathcal{E}(n))$.
\addcontentsline{toc}{section}{References}
\bibliography{Artikel}
\bibliographystyle{alpha}

\end{document}